\documentclass[12pt,letterpaper]{amsart}
\usepackage{amsmath,amssymb,amsfonts,amsthm,amscd}
\usepackage{mathrsfs,latexsym,url,graphicx,setspace,color,cancel}
\newtheorem{theorem}{Theorem}
\newtheorem*{theorem*}{Theorem}

\newtheorem{lemma}{Lemma}

\newtheorem{question}{Question}
\theoremstyle{remark}
\newtheorem{remark}{Remark}
\newtheorem{definition}{Definition}
\newtheorem{example}{Example}

\newcommand{\abs}[1]{\left\lvert#1\right\rvert}
\newcommand{\norm}[1]{\lvert\lvert#1\rvert\rvert}

\newcommand{\R}{\mathbb{R}}

\newcommand{\disc}{\mathbb{D}}
\newcommand{\C}{\mathbb{C}}
\newcommand{\n}{\mathbb{N}}

\newcommand{\zb}{\overline{z}}
\newcommand{\D}{\Omega}

\newcommand{\dbar}{\overline{\partial}}

\usepackage{hyperref}
\onehalfspace

\title[Hilbert-Schmidt Hankel Operators on Complex Ellipsoids]{Hilbert-Schmidt Hankel Operators with anti-holomorphic symbols on Complex Ellipsoids}

\author{Mehmet \c{C}el\.ik}
\address[Mehmet \c{C}elik]{University of North Texas at Dallas, Department of
Mathematics and Information Sciences,  7400 University Hills Blvd., 
Dallas, TX 75241}
\email{mehmet.celik@unt.edu}

\author{Yunus E. Zeytuncu}
\address[Yunus E. Zeytuncu]{Texas A\&M University, Department of
Mathematics, College Station, TX 77843}
\email{zeytuncu@math.tamu.edu}

\subjclass[2000]{Primary 32W05; Secondary 46B35}
\keywords{Hankel operators, Hilbert-Schmidt operators, complex ellipsoids, Canonical solution operator for $\dbar$-problem}
\date{\today}

\begin{document}

\begin{abstract}
On a complex ellipsoid in $\C^n$, we show that there is no nonzero Hankel operator with an anti-holomorphic symbol that is Hilbert-Schmidt.
\end{abstract}

\maketitle 

\section{Introduction}
\subsection{Setup and Problem}
Let $\D$ be a \textit{complex ellipsoid} in $\C^{n}$ of the form
\[\D=\{(z^{\prime},z_{n})\in \C^{n}\ |\ z^{\prime}\in\disc^{n-1}\ \text{ and }\ |z_n|<e^{-\varphi(z^{\prime})}\}\]
where $z^{\prime}=(z_1,\ldots,z_{n-1})$, $\disc^{n-1}$ is the \textit{unit $(n-1)$-disc}, $\{m_j\}_{j=1}^{n}\subset \R^{+}$, and  
\[\varphi(z^{\prime})=-\frac{1}{2m_{n}}\log\left(1-(|z_{1}|^{2m_1}+\ldots+|z_{n-1}|^{2m_{n-1}})\right).\]
Note that $\D$ is a bounded pseudoconvex Reinhardt domain.
We denote the space of square integrable functions and the space of square integrable holomorphic functions on $\D$ by $L^{2}(\D)$ and $A^{2}(\D)$ (the Bergman space of $\D$), respectively. The Bergman projection operator, $P$, is the orthogonal projection from $L^2(\D)$ onto $A^{2}(\D)$. It is an integral operator with the kernel called the Bergman kernel, which is denoted by $B_{\D}(z,w)$. Moreover, if $\{e_n(z)\}_{n=0}^{\infty}$ is an orthonormal basis for $A^2(\D)$ then the Bergman kernel can be represented as 
$$B_{\D}(z,w)=\sum\limits_{n=0}^{\infty}e_n(z)\overline{e_n(w)}.$$

For $f\in A^2(\D)$, the Hankel operator with the anti-holomorphic symbol $\overline{f}$ is formally defined on $A^2(\D)$ by
\[ H_{\overline{f}}(g)=(I-P)(\overline{f}g).\]

Note that this (possibly unbounded) operator is densely defined on $A^2(\D)$. It is known that $H_{\overline{f}}$ is bounded whenever $f$ is in a certain BMO space, see \cite{Beatrous95} and the references therein. 

For a multi-index $\gamma=(\gamma_1,\ldots,\gamma_n)\in \n^{n}$, we set 

\begin{align}
c_{\gamma}^{2}=\int\limits_{\D}\abs{z^{\gamma}}^{2}dV(z).
\end{align}
The set $\left\{\frac{z^{\gamma}}{c_{\gamma}}\right\}_{\gamma\in \n^{n}}$ constitutes a complete orthonormal basis for $A^{2}(\D)$. Indeed, for two distinct multi-indices  $\alpha$ and $\beta$,  we have $ \langle z^{\alpha},z^{\beta}\rangle_{\D}=0$. This follows from the radial symmetric structure of complex ellipsoid $\D$,
\begin{align*}
 \langle z^{\alpha},z^{\beta}\rangle_{\D}&=\int\limits_{\D} z^{\alpha}\overline{z}^{\beta} dV(z)
=\underbrace{\int\limits_{\disc^{n-1}} (z^{\prime})^{\alpha^{\prime}}(\overline{z}^{\prime})^{\beta^{\prime}} 
\underbrace{\int\limits_{|z_{n}|<e^{-\varphi(z^{\prime})}}z_{n}^{\alpha_{n}}\overline{z}_{n}^{\beta_{n}} \ dA(z_n)}_{\text{is zero unless }\ \alpha_{n}=\beta_{n}}\ dA(z^{\prime})}_{\text{is zero unless }\ \alpha^{\prime}=\beta^{\prime}}.
\end{align*}

\begin{definition} 
A linear bounded operator $T$ on a Hilbert space $H$ is called a \textit{Hilbert-Schmidt operator} if there is an orthonormal basis $\{\xi_{j}\}$ for $H$ such that the sum $\sum\limits_{j=1}^{\infty}\norm{T(\xi_{j})}^2$ is finite. 
\end{definition}
Note that, for any two orthonormal bases $\{\xi_{j}\}_{j=0}^{\infty}$ and $\{\nu_{j}\}_{j=0}^{\infty}$ we have 
\[\sum\limits_{j=0}^{\infty}\norm{T(\xi_{j})}^{2}
=\sum\limits_{j=0}^{\infty}\sum\limits_{k=0}^{\infty}\left|\langle T(\xi_{j}),\nu_{k}\rangle\right|^{2}=\sum\limits_{j=0}^{\infty}\sum\limits_{k=0}^{\infty}\left|\langle \xi_{j},T^{*}(\nu_{k})\rangle\right|^{2}=\sum\limits_{k=0}^{\infty}\norm{T^{*}(\nu_{k})}^{2}.\]
This shows that $T$ is a Hilbert-Schmidt operator if and only if $T^{*}$ is a Hilbert Schmidt operator. Therefore, the sum does not depend on the choice of orthonormal basis $\{\xi_{j}\}$.  Moreover, if $P_N$ is the projection onto Span$\{\xi_{j},1\leq j\leq N\}$, then $\{TP_{N}\}$ converges in operator norm to $T$ and $TP_{N}$ is finite ranked. Moreover, since $P_{N}$'s are Hilbert-Schmidt if $T$ is a compact operator then $TP_{N}$ is Hilbert-Schmidt, and we have $\norm{TP_N-T}_{\text{op.norm}}\rightarrow 0$ as $N\rightarrow\infty$, so Hilbert-Schmidt operators are dense in the space of compact operators. For more on Hilbert-Schmidt operators see \cite[Section X.]{Retherford93}.

In this paper, we investigate the following problem. On a given bounded Reinhardt domain in $\C^n$, characterize the symbols for which the corresponding Hankel operators are Hilbert-Schmidt. This question was first studied in $\C$ on the unit disc in \cite{ArazyFisherPeetre88}, see also Example \ref{Disc example} in the next section. The problem was studied on higher dimensional domains in \cite[Theorem at pg. 2]{KeheZhu90} where the author showed that when $n\geq 2$, on an $n$-dimensional complex ball there are no nonzero Hilbert-Schmidt Hankel operators (with anti-holomorphic symbols) on the Bergman space. The result was revisited in \cite{Schneider07} with a more robust approach. On more general domains in higher dimensions, the question was explored in \cite[Theorem 1.1]{KrantzLiRochberg97} where the authors extended the result \cite[Theorem at pg. 2]{KeheZhu90} to bounded pseudoconvex domains of finite type in $\C^2$ with smooth boundary. 

The same question was investigated on Cartan domains of tube type  in \cite[Section 2]{Arazy1996} and on strongly psuedoconvex domains in \cite{Li93, Peloso94}. Arazy studied the natural generalization of Hankel operators on Cartan domains (a circular, convex, irreducible bounded symmetric domains in $\C^n$) of tube type and rank $r>1$ in $\C^n$ for which $n/r$ is an integer. He showed that there is no nonzero Hilbert-Schmidt Hankel operators with anti-holomorphic symbols on those type of domains. Li and Peloso, independently, obtained the same result on strongly pseudoconvex domains with smooth boundary.

\subsection{Results}
In this paper, we show that on a complex ellipsoid in $\C^{n}$ (for $n>1$) there are no nonzero Hilbert-Schmidt Hankel operators with anti-holomorphic symbols on the space of square integrable holomorphic functions.

\begin{theorem}\label{Main}
Let $\D$ be a complex ellipsoid in $\C^n$ and $f \in A^2 (\D)$.
If the Hankel operator $H_{\overline{f}}$ is Hilbert-Schmidt on $A^2 (\D)$ then $f$ is constant.
\end{theorem}

\begin{remark}\label{RemarkNew}
The new ingredient in the proof of Theorem \ref{Main} is the estimate \eqref{estimating S_{alpha}} for the term denoted by $S_{\alpha}$. The analogue of Theorem 1 on the unit disc and on the unit $n$-ball were done in \cite{ArazyFisherPeetre88,Schneider07} respectively, by explicitly computing the term $S_{\alpha}$. Such a fully revealed computation on complex ellipsoids is not easy. However, an estimate like \eqref{estimating S_{alpha}} turns out to be enough.
\end{remark}

\begin{remark}
Theorem \ref{Main} partially generalizes the similar result \cite[Theorem 1.1]{KrantzLiRochberg97}. Namely, the result \cite[Theorem 1.1]{KrantzLiRochberg97} is valid for complex ellipsoids in $\C^{2}$ that are of finite type (that is, $m_1$ and $m_2$ in the definition of the ellipsoids are considered to be positive integers, which makes the domain to be of finite type). On the other hand, Theorem \ref{Main} does not have any restriction on the dimension and does not require any smoothness of the boundary.
\end{remark}


\section{Proof of the Theorem}

The set $\left\{\frac{z^{\gamma}}{c_{\gamma}}\right\}_{\gamma\in \n^{n}}$  is an orthonormal basis for $A^{2}(\D)$. In order to prove Theorem \ref{Main}, we will look at the sum 
\[\sum\limits_{\gamma}\left\|H_{\overline{f}}\left(\frac{z^{\gamma}}{c_{\gamma}}\right)\right\|^2\]
for $f\in A^{2}(\D)$.

Let's start with the $L^2$-norm of $H_{\overline{f}}\left(\frac{z^{\gamma}}{c_{\gamma}}\right)$,
\begin{align}\label{norm of Hankel}
 \left\|H_{\overline{f}}\left(\frac{z^{\gamma}}{c_{\gamma}}\right)\right\|^2=\left\langle H_{\overline{f}}\left(\frac{z^{\gamma}}{c_{\gamma}}\right),H_{\overline{f}}\left(\frac{z^{\gamma}}{c_{\gamma}}\right)\right\rangle=\frac{1}{c_{\gamma}^{2}}\left\langle H_{\overline{f}}^{*}H_{\overline{f}}\left(z^{\gamma}\right),z^{\gamma}\right\rangle.
\end{align}
Next, we rewrite the term $H_{\overline{f}}^{*}H_{\overline{f}}\left({z^{\gamma}}\right)$, by using the power series expansion $f(z)=\sum\limits_{\alpha}f_{\alpha}z^{\alpha}$,
\begin{align}\label{last term}
H_{\overline{f}}^{*}H_{\overline{f}}\left(z^{\gamma}\right)=\sum\limits_{\alpha,\beta}f_{\alpha}\overline{f}_{\beta}H_{\overline{z}^{\alpha}}^{*}H_{\overline{z}^{\beta}}\left(z^{\gamma}\right).
\end{align}
The single term $H_{\overline{z}^{\alpha}}^{*}H_{\overline{z}^{\beta}}\left(z^{\gamma}\right)$ can be computed further as,
\begin{align}\label{product of Hankels}
H_{\overline{z}^{\alpha}}^{*}H_{\overline{z}^{\beta}}\left(z^{\gamma}\right)
=H_{\overline{z}^{\alpha}}^{*}\left(\overline{z}^{\beta}z^{\gamma}-P(\overline{z}^{\beta}z^{\gamma})\right)
=H_{\overline{z}^{\alpha}}^{*}\left(\overline{z}^{\beta}z^{\gamma}-\frac{c_{\gamma}^{2}}{c_{\gamma-\beta}^{2}}z^{\gamma-\beta}\right).
\end{align}
Note that in the last term of \eqref{product of Hankels}, we set the denominator $c_{\gamma-\beta}^{2}=\infty$ if $\gamma-\beta$ has a negative number in it. Note also that the second equality follows by 
\begin{align*}
P(\overline{z}^{\beta}z^{\gamma})&=\int\limits_{\D}B_{\D}(z,w)\overline{w}^{\beta}w^{\gamma}dw
=\int\limits_{\D}\left(\sum\limits_{\alpha}\frac{z^{\alpha}\overline{w}^{\alpha}}{c_{\alpha}^{2}}\right)\overline{w}^{\beta}w^{\gamma}dw\\
&=\int\limits_{\D}\left(\sum\limits_{\alpha}\frac{z^{\alpha}\overline{w}^{\alpha+\beta}}{c_{\alpha}^{2}}\right)w^{\gamma}dw
=\frac{z^{\gamma-\beta}}{c_{\gamma-\beta}^2}\int\limits_{\D}|w|^{2\gamma}dw=\frac{c_{\gamma}^2}{c_{\gamma-\beta}^2}z^{\gamma-\beta}.
\end{align*}
Moreover, by imposing $H_{\phi}^{*}(g)=P(\overline{\phi}g)-\overline{\phi}P(g)$ in \eqref{product of Hankels} we obtain
\begin{align}\label{adjoint of Hankel}
H_{\overline{z}^{\alpha}}^{*}\left(\overline{z}^{\beta}z^{\gamma}-\frac{c_{\gamma}^{2}}{c_{\gamma-\beta}^{2}}z^{\gamma-\beta}\right)
&=P\left(z^{\alpha}\overline{z}^{\beta}z^{\gamma}-\frac{c_{\gamma}^{2}}{c_{\gamma-\beta}^{2}}z^{\alpha+\gamma-\beta}\right)
-z^{\alpha}P\left(\overline{z}^{\beta}z^{\gamma}-\frac{c_{\gamma}^{2}}{c_{\gamma-\beta}^{2}}z^{\gamma-\beta}\right)\\
\nonumber&=\left(\frac{c_{\alpha+\gamma}^{2}}{c_{\alpha+\gamma-\beta}^{2}}
-\frac{c_{\gamma}^{2}}{c_{\gamma-\beta}^{2}}\right)z^{\alpha+\gamma-\beta}.
\end{align}
Putting \eqref{adjoint of Hankel}, \eqref{product of Hankels}, and \eqref{last term} together, we continue to calculate the $L^2$-norm of $H_{\overline{f}}\left(\frac{z^{\gamma}}{c_{\gamma}}\right)$ as follows,
\begin{align}\label{norm of Hankel - continuation}
 \left\|H_{\overline{f}}\left(\frac{z^{\gamma}}{c_{\gamma}}\right)\right\|^2&=\frac{1}{c_{\gamma}^{2}}\left\langle \sum\limits_{\alpha,\beta}\overline{f}_{\alpha}f_{\beta}\left(\frac{c_{\alpha+\gamma}^{2}}{c_{\alpha+\gamma-\beta}^{2}}
-\frac{c_{\gamma}^{2}}{c_{\gamma-\beta}^{2}}\right)z^{\alpha+\gamma-\beta},z^{\gamma}\right\rangle\\
\nonumber &=\frac{1}{c_{\gamma}^{2}} \sum\limits_{\alpha}\left|\overline{f}_{\alpha}\right|^{2}\left(\frac{c_{\alpha+\gamma}^{2}}{c_{\gamma}^{2}}
-\frac{c_{\gamma}^{2}}{c_{\gamma-\alpha}^{2}}\right)\left\langle z^{\gamma},z^{\gamma}\right\rangle
= \sum\limits_{\alpha}\left|\overline{f}_{\alpha}\right|^{2}\left(\frac{c_{\alpha+\gamma}^{2}}{c_{\gamma}^{2}}
-\frac{c_{\gamma}^{2}}{c_{\gamma-\alpha}^{2}}\right).
\end{align}

We finally obtain,
\begin{align}\label{norm of Hankel - last step}
\sum\limits_{\gamma}\left\|H_{\overline{f}}\left(\frac{z^{\gamma}}{c_{\gamma}}\right)\right\|^2
&=\sum\limits_{\gamma,\alpha}\left|\overline{f}_{\alpha}\right|^{2}\left(\frac{c_{\alpha+\gamma}^{2}}{c_{\gamma}^{2}}
-\frac{c_{\gamma}^{2}}{c_{\gamma-\alpha}^{2}}\right)\\
\nonumber &=\sum\limits_{\alpha}\left|\overline{f}_{\alpha}\right|^{2}\sum\limits_{\gamma}\left(\frac{c_{\alpha+\gamma}^{2}}{c_{\gamma}^{2}}
-\frac{c_{\gamma}^{2}}{c_{\gamma-\alpha}^{2}}\right).
\end{align}

The term $\sum\limits_{\gamma}\left(\frac{c_{\gamma+\alpha}^{2}}{c_{\gamma}^2}-\frac{c_{\gamma}^{2}}{c_{\gamma-\alpha}^2}\right)$ in the identity \eqref{norm of Hankel - last step} plays an essential role in the rest of the proof, and we label it as,  
\begin{align}\label{S-alpha}
S_{\alpha}:=\sum\limits_{\gamma}\left(\frac{c_{\gamma+\alpha}^{2}}{c_{\gamma}^2}-\frac{c_{\gamma}^{2}}{c_{\gamma-\alpha}^2}\right).
\end{align}
Note that, the Cauchy-Schwarz inequality guarantees that $$\frac{c_{\gamma+\alpha}^{2}}{c_{\gamma}^2}-\frac{c_{\gamma}^{2}}{c_{\gamma-\alpha}^2}\geq 0$$
for all $\alpha$ and $\gamma$.

The computations above hold on any domains where the monomials form an orthonormal basis for the Bergman space. In particular, similar computations were done in \cite[Section 3]{ArazyFisherPeetre88} on the unit disc of the complex plane and the following result was obtained.  

\begin{example}\label{Disc example}
Let $\D=\disc$ be the unit disc in the complex plane and $f\in A^2(\disc)$. The Hankel operator $H_{\overline{f}}$ is Hilbert-Schmidt if and only if the complex derivative of $f$ is in $L^{2}(\disc)$, (that is, $f$ is in the Dirichlet space).  We present the short argument of this statement for the reader.

On the unit disc $\disc$, we can compute the coefficients $c_{\gamma}$'s explicitly and obtain that $$S_{\alpha}=\sum\limits_{\gamma}\left(\frac{c_{\alpha+\gamma}^{2}}{c_{\gamma}^{2}} -\frac{c_{\gamma}^{2}}{c_{\gamma-\alpha}^{2}}\right)=\alpha.$$ 
Therefore, by \eqref{norm of Hankel - last step} we get 
$$\sum\limits_{\gamma}\left\|H_{\overline{f}}\left(\frac{z^{\gamma}}{c_{\gamma}}\right)\right\|^2=\sum\limits_{\alpha}\alpha\left|\overline{f}_{\alpha}\right|^{2}.$$ 
On the other hand,
\[\sum\limits_{\alpha}\alpha\left|\overline{f}_{\alpha}\right|^{2}\approx \int\limits_{\disc}\left|f^{\prime}(z)\right|^{2}dA(z).\] 
This concludes that $H_{\overline{f}}$ is Hilbert-Schmidt if and only if $f$ is in the Dirichlet space. 
\end{example}

Now, we go back to the proof of the main results and estimate the term $S_{\alpha}$ on complex ellipsoids. Our goal is to show that $S_{\alpha}$ diverges for all nonzero $\alpha$ on these domains. By \eqref{norm of Hankel - last step}, this will be sufficient to conclude Theorem \ref{Main}.

In earlier results in the literature, $S_{\alpha}$'s were computed explicitly to obtain the divergence. Here, we obtain the divergence by using the following estimate from Lemma \ref{LowerBound}.

\begin{lemma}\label{LowerBound}
For any sufficiently large $N$, we have

\begin{align}\label{estimating S_{alpha}}
 S_{\alpha}\geq \sum\limits_{|\gamma|=N} \frac{c_{\gamma+\alpha}^{2}}{c_{\gamma}^2}
\end{align}
for any nonzero $\alpha$.
\end{lemma}

\begin{proof}
Considering the definition of $S_{\alpha}$ at \eqref{S-alpha}, for any sufficiently large $N$, we have
\begin{align*}
 S_{\alpha}&\geq \sum\limits_{|\gamma|\leq N} \left(\frac{c_{\gamma+\alpha}^{2}}{c_{\gamma}^{2}} -\frac{c_{\gamma}^{2}}{c_{\gamma-\alpha}^{2}}\right)\\
&=\sum\limits_{|\gamma|=N} \frac{c_{\gamma+\alpha}^{2}}{c_{\gamma}^2}-\sum\limits_{|\gamma|=N} \frac{c_{\gamma}^{2}}{c_{\gamma-\alpha}^2}+\sum\limits_{|\gamma|< N} \left(\frac{c_{\gamma+\alpha}^{2}}{c_{\gamma}^{2}} -\frac{c_{\gamma}^{2}}{c_{\gamma-\alpha}^{2}}\right).
\end{align*}
We finish the proof of the lemma when we verify that
\begin{align*}
\sum\limits_{|\gamma|< N} \left(\frac{c_{\gamma+\alpha}^{2}}{c_{\gamma}^{2}} -\frac{c_{\gamma}^{2}}{c_{\gamma-\alpha}^{2}}\right)-\sum\limits_{|\gamma|=N} \frac{c_{\gamma}^{2}}{c_{\gamma-\alpha}^2}\geq 0.
\end{align*}  
However, this simply follows from the telescoping structure of the series, see Figure \ref{Fig1}.

\begin{figure}[ht!]
\centering
\includegraphics[width=100mm]{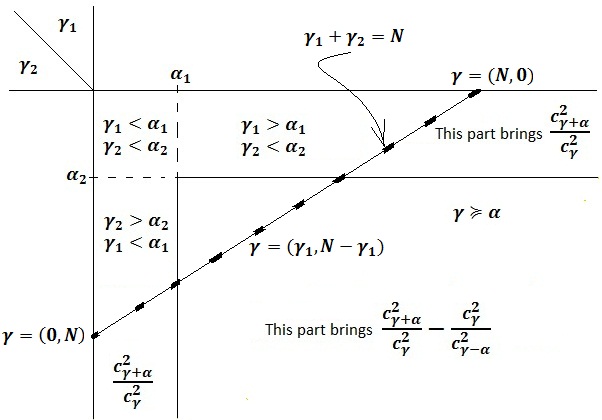}
\caption{Two dimensional illustration for distribution of terms in the sum according to $\gamma\succeq\alpha$, $\gamma\not\succeq\alpha$, and on the diagonal $|\gamma|=N$. Here, $\gamma\succeq\alpha$ means $\gamma_i\geq\alpha_i$ for $i=1,\dots,n$ and $\gamma\not\succeq\alpha$ means otherwise.}
\label{Fig1}
\end{figure}
Indeed,
\begin{align*}
\sum\limits_{|\gamma|< N} \left(\frac{c_{\gamma+\alpha}^{2}}{c_{\gamma}^{2}} -\frac{c_{\gamma}^{2}}{c_{\gamma-\alpha}^{2}}\right)-\sum\limits_{|\gamma|=N} \frac{c_{\gamma}^{2}}{c_{\gamma-\alpha}^2}&=\sum\limits_{|\gamma|< N}\frac{c_{\gamma+\alpha}^{2}}{c_{\gamma}^{2}} -\sum\limits_{|\gamma|\leq N} \frac{c_{\gamma}^{2}}{c_{\gamma-\alpha}^2}\\
&=\sum\limits_{|\gamma|< N}\frac{c_{\gamma+\alpha}^{2}}{c_{\gamma}^{2}} -\sum\limits_{\substack{|\gamma|\leq N\\ \gamma\succeq \alpha}} \frac{c_{\gamma}^{2}}{c_{\gamma-\alpha}^2}
\end{align*}  
since $\frac{c_{\gamma}^{2}}{c_{\gamma-\alpha}^2}=0$ when $\gamma\not\succeq\alpha$. Therefore, there are fewer number of negative terms and all the negative terms are canceled by some of the earlier positive terms, so the sum contains some number of positive terms after cancellation.

\end{proof}

Next, we show that the sum $\sum\limits_{|\gamma|=N} \frac{c_{\gamma+\alpha}^{2}}{c_{\gamma}^2}$ diverges for any nonzero $\alpha$. We start with computing the coefficients $c_{\gamma}$'s. For the rest of the proof, we concentrate on complex ellipsoids in $\mathbb{C}^2$. The general case follows from repeating the same arguments.

\begin{lemma}
Let $\D$ be a complex ellipsoid in $\mathbb{C}^2$, $\D=\{(z_1,z_{2})\in \C^{2}\ |\ |z_{1}|^{2m_1}+|z_{2}|^{2m_{2}}<1\}$,
where  $m_{1}, m_2\in \R^{+}$. Then 
\begin{align}
c_{\alpha}=\frac{\pi^{2}}{m_1m_2}\cdot\frac{\Gamma\left(\frac{\alpha_1+1}{m_1}\right)\cdot \Gamma\left(\frac{\alpha_2+1}{m_2}\right)}{\Gamma\left(1+\frac{\alpha_1+1}{m_1}+\frac{\alpha_2+1}{m_2}\right)}.
\end{align}
\end{lemma}

\begin{proof}
We compute the coefficients as follows,

\begin{align*}
c_\alpha^{2}&=\int\limits_{\D}\abs{z^{\alpha}}^2dV(z)
=\int\limits_{\disc}\abs{z_1}^{2\alpha_1}\int\limits_{\abs{z_{2}}<(1-\abs{z_1}^{2m_1})^{\frac{1}{2m_2}}}\abs{z_{2}}^{2\alpha_2}dA(z_{2})dA(z_{1})\\
&=\int\limits_{\disc}\abs{z_1}^{2\alpha_1}\frac{2\pi}{2\alpha_2+2}(1-\abs{z_1}^{2m_1})^{\frac{2\alpha_2+2}{2m_2}}dA(z_1)\\
&=\frac{\pi}{\alpha_2+1}\int\limits_{\disc}\abs{z_1}^{2\alpha_1}(1-\abs{z_1}^{2m_1})^{\frac{\alpha_2+1}{m_2}}dA(z_1)
=\frac{2\pi^2}{\alpha_2+1}\int\limits_{0}^{1}r^{2\alpha_1+1}(1-r^{2m_1})^{\frac{\alpha_2+1}{m_2}}dr.
\end{align*}
We continue with $r^2=r$ substitution,
\begin{align*}
c_\alpha^{2}&=\frac{\pi^2}{\alpha_2+1}\int\limits_{0}^{1}r^{\alpha_1}(1-r^{m_1})^{\frac{\alpha_2+1}{m_2}}dr.
\end{align*}
Let's substitute $t=r^{m_1}$, so $dt=m_1 r^{m_1-1}dr$. 
\begin{align*}
&=\frac{\pi^2}{\alpha_2+1}\int\limits_{0}^{1}t^{\frac{\alpha_1}{m_1}}(1-t)^{\frac{\alpha_2+1}{m_2}}\frac{1}{m_1 t^{1-\frac{1}{m_1}}}dt
=\frac{\pi^2}{(\alpha_2+1)m_1}\int\limits_{0}^{1}t^{\frac{\alpha_1+1}{m_1}-1}(1-t)^{\frac{\alpha_2+1}{m_2}}dt\\
&=\frac{\pi^2}{m_1m_2}\frac{1}{\left(\frac{\alpha_2+1}{m_2}\right)}B\left(\frac{\alpha_1+1}{m_1},\frac{\alpha_2+1}{m_2}+1\right)
\end{align*}
where $B(\cdot,\cdot)$ is the Beta function. Therefore,
\begin{align*}
c_{\alpha}^2&=\frac{\pi^2}{m_1m_2}\frac{1}{\left(\frac{\alpha_2+1}{m_2}\right)}
\frac{\Gamma\left(\frac{\alpha_1+1}{m_1}\right)\Gamma\left(\frac{\alpha_2+1}{m_2}+1\right)}{\Gamma\left(1+\frac{\alpha_1+1}{m_1}+\frac{\alpha_2+1}{m_2}\right)}
\end{align*}
By using one of the properties of Gamma function, $\Gamma(x+1)=x\Gamma(x)$ we get
 \begin{align*}
c_{\alpha}^2&=\frac{\pi^2}{m_1m_2}\frac{\Gamma\left(\frac{\alpha_1+1}{m_1}\right)\Gamma\left(\frac{\alpha_2+1}{m_2}\right)}{\Gamma\left(1+\frac{\alpha_1+1}{m_1}+\frac{\alpha_2+1}{m_2}\right)}.
\end{align*}
\end{proof}

Thus, on a complex ellipsoid 
\begin{align}\label{the c-ratio}
\frac{c_{\gamma+\alpha}^2}{c_{\gamma}^{2}}&=\frac{\Gamma\left(\frac{\gamma_1+\alpha_1+1}{m_1}\right) 
\Gamma\left(\frac{\gamma_2+\alpha_2+1}{m_2}\right) 
\Gamma\left(1+\frac{\gamma_1+1}{m_1}+\frac{\gamma_2+1}{m_2}\right)}
{\Gamma\left(\frac{\gamma_1+1}{m_1}\right) 
\Gamma\left(\frac{\gamma_2+1}{m_2}\right)
\Gamma\left(1+\frac{\gamma_1+\alpha_1+1}{m_1}+\frac{\gamma_2+\alpha_2+1}{m_2}\right)}.
\end{align}
By Lemma \ref{LowerBound}, we have 
\begin{align}\label{first step on estimating S_{alpha}}
 S_{\alpha}\geq \sum\limits_{|\gamma|=N} \frac{c_{\gamma+\alpha}^{2}}{c_{\gamma}^2}
=\sum\limits_{\gamma_1+\gamma_2=N} \frac{c_{(\gamma_1,\gamma_2)+\alpha}^{2}}{c_{(\gamma_1,\gamma_2)}^2}
=\sum\limits_{\gamma_1=0}^{N} \frac{c_{(\gamma_1,N-\gamma_1)+\alpha}^{2}}{c_{(\gamma_1,N-\gamma_1)}^2}
=\sum\limits_{k=0}^{N} \frac{c_{\alpha+(k,N-k)}^{2}}{c_{(k,N-k)}^2}
\end{align}
for any nonzero $\alpha$. Note that, in the last equation we replace $\gamma_1$ by $k$ to simplify the notation. We insert the ratio \eqref{the c-ratio} into \eqref{first step on estimating S_{alpha}} to get

\begin{align}\label{second step on estimating S_{alpha}}
 S_{\alpha}\geq \sum\limits_{k=\frac{N}{3}}^{\frac{2N}{3}}
\frac{\Gamma\left(\frac{k+\alpha_1+1}{m_1}\right) 
\Gamma\left(\frac{N-k+\alpha_2+1}{m_2}\right) 
\Gamma\left(1+\frac{k+1}{m_1}+\frac{N-k+1}{m_2}\right)}
{\Gamma\left(\frac{k+1}{m_1}\right) 
\Gamma\left(\frac{N-k+1}{m_2}\right)
\Gamma\left(1+\frac{k+\alpha_1+1}{m_1}+\frac{N-k+\alpha_2+1}{m_2}\right)} 
\end{align}

Note that, when $N$ is large enough each $\Gamma(\cdots)$-expression in \eqref{second step on estimating S_{alpha}} is big enough to invoke \textit{Stirling's approximation} $\Gamma(x)\sim e^{-x}x^{x-\frac{1}{2}}$ to get

\begin{align}\label{third step on estimating S_{alpha}}
S_{\alpha}\geq c\sum\limits_{k=\frac{N}{3}}^{\frac{2N}{3}}
\Bigg\{&\frac{e^{-\left(\frac{k+\alpha_1+1}{m_1}+\frac{N-k+\alpha_2+1}{m_2}+1+\frac{k+1}{m_1}+\frac{N-k+1}{m_2}\right)}}
{e^{-\left(\frac{k+1}{m_1}+\frac{N-k+1}{m_2}+1+\frac{k+\alpha_1+1}{m_1}+\frac{N-k+\alpha_2+1}{m_2}\right)}}
\cdot \frac{\sqrt{\frac{k+1}{m_1}}\sqrt{\frac{N-k+1}{m_2}}\sqrt{1+\frac{k+\alpha_1+1}{m_1}+\frac{N-k+\alpha_2+1}{m_2}}}
{\sqrt{\frac{k+\alpha_1+1}{m_1}}\sqrt{\frac{N-k+\alpha_2+1}{m_2}}\sqrt{1+\frac{k+1}{m_1}+\frac{N-k+1}{m_2}}}\\
\nonumber &\\
\nonumber&\cdot\frac{\left(\frac{k+\alpha_1+1}{m_1}\right)^{\frac{k+\alpha_1+1}{m_1}}\left(\frac{N-k+\alpha_2+1}{m_2}\right)^{\frac{N-k+\alpha_2+1}{m_2}}\left(1+\frac{k+1}{m_1}+\frac{N-k+1}{m_2}\right)^{1+\frac{k+1}{m_1}+\frac{N-k+1}{m_2}}}
{\left(\frac{k+1}{m_1}\right)^{\frac{k+1}{m_1}}\left(\frac{N-k+1}{m_2}\right)^{\frac{N-k+1}{m_2}}\left(1+\frac{k+\alpha_1+1}{m_1}+\frac{N-k+\alpha_2+1}{m_2}\right)^{1+\frac{k+\alpha_1+1}{m_1}+\frac{N-k+\alpha_2+1}{m_2}}}\Bigg\}
\end{align}
 
From the last estimate in \eqref{third step on estimating S_{alpha}} it is easy to see that for large enough $N$ we have

\begin{align}\label{forth step on estimating S_{alpha}}
S_{\alpha}&\geq c\sum\limits_{k=\frac{N}{3}}^{\frac{2N}{3}}
\left\{\frac{\left(\frac{k}{m_1}\right)^{\frac{k+\alpha_1+1}{m_1}}\left(\frac{N-k}{m_2}\right)^{\frac{N-k+\alpha_2+1}{m_2}}\left(\frac{k}{m_1}+\frac{N-k}{m_2}\right)^{1+\frac{k+1}{m_1}+\frac{N-k+1}{m_2}}}
{\left(\frac{k}{m_1}\right)^{\frac{k+1}{m_1}}\left(\frac{N-k}{m_2}\right)^{\frac{N-k+1}{m_2}}\left(\frac{k}{m_1}+\frac{N-k}{m_2}\right)^{1+\frac{k+\alpha_1+1}{m_1}+\frac{N-k+\alpha_2+1}{m_2}}}\right\}
\end{align}
\begin{align*}
\nonumber&=c^{\prime}\sum\limits_{k=\frac{N}{3}}^{\frac{2N}{3}} \frac{k^{\frac{\alpha_1}{m_1}}(N-k)^{\frac{\alpha_2}{m_2}}}{N^{\frac{\alpha_1}{m_1}+\frac{\alpha_2}{m_2}}}=c^{\prime\prime}\sum\limits_{k=\frac{N}{3}}^{\frac{2N}{3}} 1=c^{\prime\prime\prime}\cdot N
\end{align*}
Hence, $S_{\alpha}$ diverges for every nonzero $\alpha$.\\

Note once again that we obtain the divergence of $S_{\alpha}$'s without explicitly computing them but instead estimating them from below as, 
\begin{equation}\label{finalestimate}
S_{\alpha}\geq\sum\limits_{|\gamma|=N}\frac{c_{\gamma+\alpha}^{2}}{c_{\gamma}^{2}}=\sum\limits_{k=0}^{N}\frac{c_{\alpha+(k,N-k)}^{2}}{c_{(k,N-k)}^{2}}\geq \sum\limits_{k=\frac{N}{3}}^{\frac{2N}{3}}\frac{c_{\alpha+(k,N-k)}^{2}}{c_{(k,N-k)}^{2}}.
\end{equation}
Then, we notice that for large enough $N$ and $k\in\left[\frac{N}{3},\frac{2N}{3}\right]$ the fraction $\frac{c_{\alpha+(k,N-k)}^{2}}{c_{(k,N-k)}^{2}}$ is almost like $1$. Therefore, the last sum is of the size $N$ and $S_{\alpha}$ diverges. This careful estimate is the main ingredient that enables us to extend the result in \cite{KeheZhu90,Schneider07}.

\section{Remarks and Questions}

The proof of Theorem \ref{Main} works on some other Reinhardt domains, too. Below we go over the computations on polydiscs.

\subsection{Computations on Polydiscs}
We go over the argument on bi-disc in $\C^2$, see also  \cite{KrantzLiRochberg97}. A straightforward computation indicates,
$$c_{\alpha}^2=\int\limits_{\mathbb{D}\times\mathbb{D}}|z^{\alpha}|^2dV(z)=\frac{\pi^2}{(\alpha_1+1)(\alpha_2+1)}.$$
By invoking the estimate \eqref{finalestimate}, we obtain
\begin{align*}
S_{\alpha}&\geq \sum\limits_{k=\frac{N}{3}}^{\frac{2N}{3}}\frac{c_{\alpha+(k,N-k)}^{2}}{c_{(k,N-k)}^{2}}=\sum\limits_{k=\frac{N}{3}}^{\frac{2N}{3}}\frac{(k+1)(N-k+1)}{(k+1+\alpha_1)(N-k+1+\alpha_2)}\\
&\geq \sum\limits_{k=\frac{N}{3}}^{\frac{2N}{3}}\frac{(\frac{N}{3}+1)(\frac{N}{3}+1)}{(\frac{2N}{3}+1+\alpha_1)(\frac{2N}{3}+1+\alpha_2)}\\
&\geq \sum\limits_{k=\frac{N}{3}}^{\frac{2N}{3}}\frac{1}{5}\hskip 1cm \text{ (for sufficiently large }N)\\
&=\frac{N}{15}
\end{align*}
for all nonzero multi-index $\alpha$. Hence, $S_{\alpha}$ diverges for any nonzero $\alpha$ and we obtain Theorem 1 on polydiscs.
 
\subsection{Open Problem}
In the light of Theorem 1, and the computations above, the following question arises. 
\begin{question}
Is there a bounded Reinhardt domain in $\C^n$ on which there exists a non-trivial Hilbert-Schmidt Hankel operator with a conjugate holomorphic symbol?
\end{question}

Note that the boundedness assumption in the question is essential. Indeed, when the Bergman space is finite dimensional then any Hankel operator is Hilbert-Schmidt and in \cite{Wiegerinck84} Wiegerinck constructed Reinhardt domains (unbounded but with finite volume) in $\C^2$ for which the Bergman spaces are finite dimensional. 

\subsection{Canonical solution operator for $\dbar$-problem:}
The canonical solution operator for $\dbar$-problem restricted to $(0,1)$-forms with holomorphic coefficients is not a Hilbert-Schmidt operator on complex ellipsoids because the canonical solution operator for $\dbar$-problem restricted to $(0,1)$-forms with holomorphic coefficients is a sum of Hankel operators with $\{\zb_{j}\}_{j=1}^{n}$ as symbols (by Theorem \ref{Main} such Hankel operators are not Hilbert-Schmidt), $$\dbar^*N_1(g)=\dbar^*N_1 \left(\sum\limits_{j=1}^{n}g_{j}d\zb_{j}\right)=\sum\limits_{j=1}^{n}H_{\zb_j}(g_{j})$$ 
for any $(0,1)$-form $g$ with holomorphic coefficients.
Similar analysis has been done on unit disc, unit $n$-ball, and polydisc by Haslinger in \cite{Haslinger01}.

\section{acknowledgement}
We would like to thank Friedrich Haslinger and S\"{o}nmez \c{S}ahuto\u{g}lu for valuable comments on a
preliminary version of this manuscript.

\bibliographystyle{amsalpha}
\bibliography{CelikZeytuncuBib}

\providecommand{\bysame}{\leavevmode\hbox to3em{\hrulefill}\thinspace}
\providecommand{\MR}{\relax\ifhmode\unskip\space\fi MR }
\providecommand{\MRhref}[2]{%
  \href{http://www.ams.org/mathscinet-getitem?mr=#1}{#2}
}
\providecommand{\href}[2]{#2}
\begin{thebibliography}{KLR97}

\bibitem[AFP88]{ArazyFisherPeetre88}
J.~Arazy, S.~D. Fisher, and J.~Peetre, \emph{Hankel operators on weighted
  {B}ergman spaces}, Amer. J. Math. \textbf{110} (1988), no.~6, 989--1053.

\bibitem[Ara96]{Arazy1996}
Jonathan Arazy, \emph{Boundedness and compactness of generalized {H}ankel
  operators on bounded symmetric domains}, J. Funct. Anal. \textbf{137} (1996),
  no.~1, 97--151.

\bibitem[BL95]{Beatrous95}
Frank Beatrous and Song-Ying Li, \emph{Trace ideal criteria for operators of
  {H}ankel type}, Illinois J. Math. \textbf{39} (1995), no.~4, 723--754.

\bibitem[Has01]{Haslinger01}
F.~Haslinger, \emph{Compactness of the canonical solution operator to
  {$\overline\partial$} restricted to {B}ergman spaces}, Functional-analytic
  and complex methods, their interactions, and applications to partial
  differential equations ({G}raz, 2001), World Sci. Publ., River Edge, NJ,
  2001, pp.~394--400.

\bibitem[KLR97]{KrantzLiRochberg97}
Steven~G. Krantz, Song-Ying Li, and Richard Rochberg, \emph{The effect of
  boundary geometry on {H}ankel operators belonging to the trace ideals of
  {B}ergman spaces}, Integral Equations Operator Theory \textbf{28} (1997),
  no.~2, 196--213.

\bibitem[Li93]{Li93}
Huiping Li, \emph{Schatten class {H}ankel operators on the {B}ergman spaces of
  strongly pseudoconvex domains}, Proc. Amer. Math. Soc. \textbf{119} (1993),
  no.~4, 1211--1221. \MR{1169879 (94a:47045)}

\bibitem[Pel94]{Peloso94}
Marco~M. Peloso, \emph{Hankel operators on weighted {B}ergman spaces on
  strongly pseudoconvex domains}, Illinois J. Math. \textbf{38} (1994), no.~2,
  223--249. \MR{1260841 (95e:47039)}

\bibitem[Ret93]{Retherford93}
J.~R. Retherford, \emph{Hilbert space: compact operators and the trace
  theorem}, London Mathematical Society Student Texts, vol.~27, Cambridge
  University Press, Cambridge, 1993.

\bibitem[Sch07]{Schneider07}
G.~Schneider, \emph{A different proof for the non-existence of
  {H}ilbert-{S}chmidt {H}ankel operators with anti-holomorphic symbols on the
  {B}ergman space}, Aust. J. Math. Anal. Appl. \textbf{4} (2007), no.~2, Art.
  1, 7.

\bibitem[Wie84]{Wiegerinck84}
Jan J. O.~O. Wiegerinck, \emph{Domains with finite-dimensional {B}ergman
  space}, Math. Z. \textbf{187} (1984), no.~4, 559--562.

\bibitem[Zhu90]{KeheZhu90}
Ke~He Zhu, \emph{Hilbert-{S}chmidt {H}ankel operators on the {B}ergman space},
  Proc. Amer. Math. Soc. \textbf{109} (1990), no.~3, 721--730.

\end{thebibliography}
\end{document}